\documentclass[11pt]{amsart}
\usepackage{amsmath}
\usepackage{amsthm}
\usepackage{amssymb}
\usepackage{a4wide}
\usepackage{clock}

\usepackage{hyperref}
\usepackage{amsfonts}

\newtheorem{lemma}{Lemma}[section]

\newtheorem{theorem}{Theorem}[section]

\newtheorem{remark}{Remark}[section]

\def\rr{\mathbb{R}}

\def\rr{\mathbb{R}}

\def\eps{\varepsilon}

\def\R{\right}

\def\rrd{\rr^d}

\def\R{\mathbb R}

\begin{document}

\title[First eigenvalue for nonlocal problems]
{Lower and upper bounds for the first eigenvalue of nonlocal
diffusion problems in the whole space}

\author[L. I. Ignat, J. D. Rossi and A. San Antolin]{Liviu I. Ignat, Julio D. Rossi and Angel San Antolin}

\address{L. I. Ignat
\hfill\break\indent Institute of Mathematics ``Simion Stoilow'' of
the Romanian Academy, \hfill\break\indent 21 Calea Grivitei Street
 \hfill\break\indent 010702, Bucharest,  ROMANIA \hfill\break\indent
and \hfill\break\indent
 BCAM - Basque Center for Applied Mathematics, \hfill\break\indent
Bizkaia Technology Park, Building 500 Derio, \hfill\break\indent
Basque Country, SPAIN.}
 \email{{\tt
liviu.ignat@gmail.com}\hfill\break\indent  {\it Web page: }{\tt
http://www.imar.ro/\~\,lignat}}

\address{J. D. Rossi
\hfill\break\indent Departamento de An\'{a}lisis Matem\'{a}tico,
Universidad de Alicante, \hfill\break\indent Ap. correos 99, 03080,
\hfill\break\indent Alicante, SPAIN. \hfill\break\indent On leave
from \hfill\break\indent Dpto. de Matem{\'a}ticas, FCEyN,
Universidad de Buenos Aires, \hfill\break\indent 1428, Buenos Aires,
ARGENTINA. } \email{{\tt jrossi@dm.uba.ar} \hfill\break\indent {\it
Web page: }{\tt http://mate.dm.uba.ar/$\sim$jrossi/}}

\address{A. San Antolin
\hfill\break\indent Departamento de An\'{a}lisis Matem\'{a}tico,
 Universidad de Alicante,
\hfill\break\indent Ap. correos 99, \hfill\break\indent 03080,
Alicante, SPAIN. } \email{{\tt angel.sanantolin@ua.es}}


\begin{abstract} We find lower and upper bounds for the first eigenvalue of a nonlocal
diffusion operator of the form $ T(u) = - \int_{\rr^d} K(x,y)
(u(y)-u(x))
\, dy$. Here we consider a kernel $K(x,y)=\psi (y-a( x))+\psi(x-a( y))$
where $\psi$ is a bounded, nonnegative function supported in the unit
ball  and $a$ means a diffeomorphism on $\rr^d$. A simple
example being a linear function $a(x)= Ax$. The upper and lower
bounds that we obtain are given in terms of the Jacobian of $a$
and the integral of $\psi$. Indeed, in the linear case $a(x) = Ax$
we obtain an explicit expression for the first eigenvalue  in the whole $\rr^d$ and it is positive when the
the determinant  of the matrix $A$ is different from
one. As an application of our results, we observe that, when the
first eigenvalue is positive, there is an exponential decay for
the solutions to the associated evolution problem. As a tool to
obtain the result, we also study the behaviour of the principal
eigenvalue of the nonlocal Dirichlet  problem in the ball $B_R$ and prove that
it converges to the first eigenvalue in the whole space as $R\to
\infty$.
\end{abstract}

\maketitle

\section{Introduction}
\label{Sect.intro}
\setcounter{equation}{0}

Nonlocal problems have been recently widely used to model
diffusion processes. When $u(x,t)$ is interpreted as the density
of a single population at the point $x$ at time $t$ and $J(x-y)$
is the probability of ``jumping" from location $y$ to location
$x$, the convolution $(J*u)(x) = \int_{\R^d} J(y-x) u(y,t)\, dy$
is the rate at which individuals arrive to position $x$ from all
other positions, while $-\int_{\R^d} J(y-x) u(x,t)\, dy$ is the
rate at which they leave position $x$ to reach any other position.
If in addition no external source is present, we obtain that $u$
is a solution to the following evolution problem
\begin{equation}\label{probgen}
u_t(x,t) = \displaystyle \int_{\R^d} J(y-x) (u(y,t)-u(x,t))\, dy.
\end{equation}
This equation is understood to hold in a bounded domain, this is,
for $x\in
\Omega$ and has to be complemented with a ``boundary" condition.
For example,
$u=0$ in $\R^d\setminus
\Omega$ which means that the habitat $\Omega$ is surrounded by a hostile
environment (see
 \cite{F} and \cite{Du} for a general nonlocal vector calculus).
Problem (\ref{probgen}) and its stationary version have been
considered recently in connection with real applications (for
example to peridynamics, a recent model for elasticity), we quote
for instance
\cite{AMRT},
\cite{CERW1},
\cite{CERW2}, \cite{BCC},
\cite{BFRW},
 \cite{Co1},
\cite{CD2},
\cite{CR1},
\cite{CCR}, \cite{PLPS}, \cite{Sill}, \cite{SL}
and the recent book \cite{libro}. See also \cite{IR2} for the appearance of convective
terms, \cite{AMRT2} for a problem with nonlinear nonlocal
diffusion and \cite{CCEM}, \cite{CER} for other features in
related nonlocal problems.

On the other hand, it is well known that eigenvalue problems are a
fundamental tool to deal with local problems. In particular, the
so-called principal eigenvalue of the Laplacian with Dirichlet
boundary conditions,
\begin{equation}\label{eig-local}
\left\{
\begin{array}{ll}
-\Delta v(x)= \sigma v(x), \qquad & x\in \Omega,\\
\quad v(x)=0, \qquad & x\in \partial\Omega,
\end{array}
\right.
\end{equation}
plays an important role, since it gives the exponential decay of
solutions to the associated parabolic problem, $u_t = \Delta u$
with $u|_{\partial \Omega}=0$. The properties of the principal
eigenvalue of (\ref{eig-local}) are well-known, see \cite{GT}.

For the nonlocal problem, in \cite{GMR-autov} the authors consider
the ``Dirichlet" eigenvalue problem for a nonlocal operator in a
smooth bounded domain $\Omega$, that is,
\begin{equation}\label{eigen}
\left\{
\begin{array}{ll}
\displaystyle (J*u)(x) -u(x) = -\lambda u(x), \quad & x\in \Omega,\\[0.25pc]
\qquad u (x) = 0, \qquad & x\in \R^d\setminus \Omega.
\end{array}
\right.
\end{equation}
They show that the first eigenvalue has associated a positive
eigenfunction and that the eigenvalue goes to zero as the domain is
expanded, i.e., $\lambda_1 (k\Omega) \to 0$ as $k \to \infty$. In
addition, it is proved in \cite{CCR} that solutions to
\eqref{probgen} in the whole $\rr^d$ decay in the $L^2$-norm as
$t^{-d/4}$. Therefore the first eigenvalue in the whole space $\R^d$
is zero for the convolution case. When we face a convolution one of
the main tools is the use of the Fourier transform, see \cite{CCR}.

For more general kernels, in \cite{IR3} energy methods where
applied to obtain decay estimates for solutions to nonlocal
evolution equations whose kernel is not given by a convolution,
that is, equations of the form
\begin{equation}\label{lineal}
u_t (x,t) = \int_{\rr^d} K(x,y)(u(y,t)
 - u(x,t)) \, dy
\end{equation}
with $K(x,y)$ a symmetric nonnegative kernel. The obtained decay
estimates are of polinomial type, more precisely, $\| u (\cdot,
t)\|_{L^2 (\rr^d)} \leq C t^{-d/4}$. We remark that this decay bound
need not be optimal, in fact, in \cite{IR3} there is a particular
example of a kernel $K$ that give exponential decay in $L^2(\rr)$.
The exponential decay of solutions suggests that the associated
first eigenvalue is positive.

Our main goal in the present work is to study properties of the
principal eigenvalue of nonlocal diffusion operators when the
associated kernel is not of convolution type. Some preliminary
properties are already known, as existence, uniqueness and a
variational characterization. To go further, we need to assume
some structure for the kernel. Let us consider a function $\psi$
nonnegative, bounded and supported in the unit ball in $\rr^d$. We associate with this function a kernel of the form
\begin{equation}\label{forma.nucleo}
K(x,y)=\psi (y-a( x))+\psi(x-a( y))
\end{equation}
where $a(x)$ is a diffeomorphism on $\rr^d$. Note that $K$ is
symmetric and that the convolution type kernels also take the form
\eqref{forma.nucleo} (just put $a(x)=x$). For this kernels let us
look for the first eigenvalue of the associated nonlocal operator,
that is,
\begin{equation}\label{eigenvalue}
- \int_{\rr^d} K(x,y) (u(y)-u(x)) \, dy = \lambda_1 u (x).
\end{equation}

Some known results (that we state in the next section for
completeness) read as follows: For any bounded domain $\Omega$
there exists a principal eigenvalue  $\lambda_1 (\Omega)$ of
problem
\eqref{eigenvalue} with $u\equiv 0$ in $\rr^d \setminus \Omega$. The
 corresponding non-negative eigenfunction $\phi_1(x)$ is strictly positive in~$\Omega$.
 Moreover, the first eigenvalue is given by
\begin{equation}\label{lambdar}
\lambda_1(\Omega)=
\inf _{u\in L^2(\Omega)}
\frac {\displaystyle\int _{\rr^d}\int_{\rrd} K(x,y)(\tilde u(x)-\tilde u(y))^2dxdy}
{\displaystyle\int_{\Omega}u^2(x)dx}.
\end{equation}
Here we have denoted by $\tilde{u}$ the extension by zero of $u$,
$$\tilde u(x)=\left\{
\begin{array}{ll}
u(x),& x\in \Omega,\\
0,& x\in \rr^d \setminus \Omega.
\end{array}
\right.$$
We will use this notation trough the whole paper. When we deal
with the whole space we have
\begin{equation}\label{lambda1}
\lambda_1 (\rr^d)=\inf _{u\in L^2(\rr^d)}\frac{\displaystyle
\int_{\rr^d}\int_{\rrd} K(x,y)(u(x)-u(y))^2dxdy}{\displaystyle\int _{\rrd}u^2(x)dx}.
\end{equation}

The main results of this paper are the following:

\begin{theorem} \label{teo.conver} Let $\Omega $ be a bounded domain with $0\in \Omega$ and
consider its dilations by a real factor $R$, $R\Omega= \{ Rx
\, : \, x\in \Omega\}$. Then
\begin{equation}\label{lambdalimit.intro}
\lambda_1(\rr^d) =\lim _{R\rightarrow \infty} \lambda_1(R \Omega).
\end{equation}
\end{theorem}

Now, we state our result concerning lower bounds for the first
eigenvalue.

\begin{theorem} \label{teo.cota.por.abajo}
Assume that the kernel is given by \eqref{forma.nucleo} and that
the Jacobian of $a^{-1}$, $J_{a^{-1}}$, verifies
 \[
 \displaystyle\sup_{x\in \rr^d}  |J_{a^{-1}}(x)|=M< 1
 \qquad \text{ or } \qquad
  \displaystyle\inf _{x\in \rr^d}|J_{a^{-1}}(x)|=m> 1.
\]
 Then
 $$
\lambda_1 (\rr^d) \geq 2(1-M^{1/2})^2 \left(\int
_{\rr^d}\psi(x)dx\right),
$$
in the first case and
$$
\lambda_1(\rr^d) \geq 2(m^{1/2}-1)^2 \left(\int_{\rr^d}\psi(x)dx\right),
$$
in the second case.
\end{theorem}

Concerning upper bounds we have the following less general result.

\begin{theorem} \label{teo.cota.por.arriba}
Let $a$ be a diffeomorphism homogeneous of degree one, that is,
$a(Rx) =R a(x)$. Assume that the kernel is given by
\eqref{forma.nucleo}.
%and let
% \[
% \displaystyle\sup_{x\in \rr^d}  |J_{a^{-1}}(x)|=M.
%\]
 Then
\begin{equation}\label{cota.por.arriba.intro}
\lambda_1 (\rr^d) \leq 2 \left(\int_{\rr^d}\psi(x)dx\right)
\inf _{\|\phi\|_{L^2(B_1)=1} } \int _{\rr^d}\big(\phi(x)-\phi(a(x))\big)^2dx,
\end{equation}
where the infimum is taken over all functions $\phi$ supported in the unit ball of $\rr^d$.
\end{theorem}

\begin{remark}
Since we can consider $\phi \geq 0$, we get
$$
\begin{array}{rl}
\displaystyle \int_{\rr^d}\big(\phi(x)-\phi(a(x))\big)^2dx &
\displaystyle = \int_{\rr^d}\phi^2(x)dx+ \int_{\rr^d} \phi^2(a(x))dx
-
2\int_{\rr^d}\phi(x)\phi(a(x))dx \\[10pt]
& \displaystyle \leq \int_{\rr^d}\phi^2(x)dx+ \int_{\rr^d}
\phi^2(a(x))dx.
\end{array}
$$
Hence, from \eqref{cota.por.arriba.intro} we immediately obtain the
following bound
\begin{equation}\label{est.sup}
\lambda_1 (\rr^d) \leq 2 \big(1+ \displaystyle\sup_{x\in \rr^d}
 |J_{a^{-1}}(x)|\big)\left(\int_{\rr^d}\psi(x)dx\right).
\end{equation}
\end{remark}

For invertible linear maps $a$ on $\R^d$ we obtain the following
sharp result.

\begin{theorem}\label{main.result}
Let $K$ be given by \eqref{forma.nucleo} with an invertible linear
map $a(x)=Ax$. Then
\begin{equation}\label{main.44}
\lambda_1 (\R^d) = \lim _{R\rightarrow \infty}\lambda_1(B_R)=
2(1-|\det(A)|^{-1/2})^2\left(\int _{\rr^d}\psi(x)dx\right).
\end{equation}
\end{theorem}

\begin{remark}
Note that for a linear function $a$ the bound
\eqref{est.sup} is not sharp. However, Theorems
\ref{teo.cota.por.abajo} and \ref{teo.cota.por.arriba} provide
lower and upper bounds for $\lambda_1 (\rr^d)$  when $M<1$ or
$m>1$ that depend linearly on $\int \psi$ in terms of the jacobian
of the diffeomorphism $a^{-1}$.
\end{remark}

As an immediate application of our results, we observe that, when
the first eigenvalue is positive, we have exponential decay for
the solutions to the associated evolution problem in $\rr^d$. In
fact, let us consider,
$$
u_t (x,t) = \int_{\rr^d} K(x,y) (u(y,t)- u(x,t))\, dy,
$$
with an initial condition $u(x,0)=u_0 (x) \in L^2 (\rr^d)$.
Multiply by $u(x,t)$ and integrate to obtain
$$
\begin{array}{rl}
\displaystyle \frac12 \frac{d}{dt} \int_{\rr^d} u^2 (x,t) \, dx & \displaystyle =
\int_{\rr^d} \int_{\rr^d} K(x,y) (u(y,t)- u(x,t)) u(x,t)\, dy\, dx
\\[8pt]
& \displaystyle = - \frac12 \int_{\rr^d} \int_{\rr^d} K(x,y)
(u(y,t)- u(x,t))^2 \, dy\, dx
\\[8pt]
& \displaystyle \leq - \frac12 \lambda_1 \int_{\rr^d} u^2 (x,t)\,
dx.
\end{array}
$$
Thus, an exponential decay of $u$ in $L^2$-norm follows
$$
\displaystyle \int_{\rr^d} u^2 (x,t) \, dx \leq
\left(\int_{\rr^d}u^2(x,0) \, dx \right) \cdot e^{-\lambda_1 t}.
$$

The paper is organized as follows: in Section~\ref{Sect.prelim} we
collect some preliminary results and prove Theorem~\ref{teo.conver};
while in Section~\ref{main.results} we collect the proofs of the
lower and upper bounds for the first eigenvalue; we prove
Theorem~\ref{teo.cota.por.abajo}, Theorem~\ref{teo.cota.por.arriba}
and Theorem~\ref{main.result}.

\section{Properties of the first eigenvalue. Proof of Theorem~\ref{teo.conver}}
\label{Sect.prelim}
\setcounter{equation}{0}

First, let us state some known properties of the first eigenvalue
of our nonlocal operator.

\begin{theorem}
For any bounded domain $\Omega$ there exists a principal
eigenvalue  $\lambda_1 (\Omega)$ of problem \eqref{eigenvalue},
i.e. the
 corresponding non-negative eigenfunction $\phi_1(x)$ is strictly positive in $\Omega$.
\end{theorem}

\begin{proof} It follows from  \cite{KR}.
\end{proof}

\begin{theorem}
The first eigenvalue of problem \eqref{eigenvalue} satisfies
\begin{equation}\label{lambdar.prelim}
\lambda_1(\Omega)=\inf _{u\in L^2(\Omega)}
\frac{\displaystyle\int _{\rr^d}\int_{\rrd}
K(x,y)(\tilde u(x)-\tilde u(y))^2\, dx\,
dy}{\displaystyle\int_{\Omega}u^2(x)\, dx}.
\end{equation}
%with
%$$\tilde u(x)=\left\{
%\begin{array}{ll}
%u(x),& x\in \Omega,\\
%0,& x\in \rr^d \setminus \Omega.
%\end{array}
%\right.$$
\end{theorem}

\begin{proof} See \cite{GMR-autov}.
\end{proof}

Now, to simplify the presentation, we prove
Theorem~\ref{teo.conver} in the special case of balls $B_R$ that
are centered at the origin with radius $R$ (we will use this
notation in the rest of the paper) and next we deduce from this
fact the general case, $\Omega$ a bounded domain.

\begin{lemma}\label{asimptoticlimit}
Let $\lambda_1(\rr^d)$ be defined by \eqref{lambda1}. Then
\begin{equation}\label{lambdalimit}
\lambda_1 (\rr^d)=\lim _{R\rightarrow \infty} \lambda_1(B_R).
\end{equation}
\end{lemma}

\begin{proof}
First of all, observe that for any $R_1\leq R_2$ we have
$B_{R_1}\subset B_{R_2}$ and then
$$\lambda_1(B_{R_1})\geq \lambda _1(B_{R_2})>0.$$
Then we deduce that there exists the limit
$$\lim _{R\rightarrow \infty} \lambda_1(B_{R})\geq 0.$$

{\textbf{Step I.}} Let us choose $u\in L^2(B_R)$. By the
definition of $\lambda_1(\rr^d)$ we get
$$\frac {\displaystyle
\int _{\rr^d}\int_{\rr^d} K(x,y)(\tilde u(x)-\tilde u(y))^2dxdy}{\displaystyle\int _{B_R}u^2(x)dx}
=\frac{\displaystyle\int _{\rr^d}\int_{\rr^d} K(x,y)(\tilde
u(x)-\tilde u(y))^2dxdy}{\displaystyle\int_{\rr^d}\tilde u^2(x)dx}
\geq \lambda_1(\rr^d).$$
Taking the infimum in the right hand side over all functions $u\in L^2(B_R)$ we obtain that for any $R>0$
\begin{equation}\label{est.1}
\lambda_1(B_R)\geq \lambda_1 (\rr^d).
\end{equation}

{\textbf{Step II.}} Let be $\eps>0$. Then there exists $u_\eps\in L^2(\rr)$ such that
\begin{equation}\label{eps.est}
\lambda_1(\rr^d)+
\eps\geq \frac {\displaystyle\int _{\rr^d}\int_{\rrd} K(x,y)(u_\eps(x)-u_\eps(y))^2dxdy}
{\displaystyle\int _{\rrd}u_\eps^2(x)dx}.
\end{equation}
We choose  $u_{\eps,R}$ defined by
$$u_{\eps,R}(x)=u_\eps (x)\chi_{B_R}(x).$$
We claim that
\begin{equation}\label{est.2}
\int _{B_R} u^2_{\eps,R}(x)dx\rightarrow \int_{\rr^d} u_\eps^2(x)dx
\end{equation}
and
\begin{equation}\label{est.3}
\int_{\rr^d}\int_{\rrd} K(x,y)( u_{\eps,R}(x)- u_{\eps,R}(y))^2\, dx\, dy\rightarrow
\int_{\rr^d}\int_{\rrd} K(x,y)( u_\eps(x)-u_\eps(y))^2\, dx\, dy.
\end{equation}
Assume these claims for the momment; using that $u_{\eps,R}$
vanishes outside the ball $B_R$ and the definition of
$\lambda_1(B_R)$ we get
$$\frac{\displaystyle\int _{\rr^d}\int_{\rrd} K(x,y)( u_{\eps,R}(x)- u_{\eps,R}(y))^2dxdy}
{\displaystyle\int _{B_R} u^2_{\eps,R}(x)dx}
\geq \lambda_1(B_R).$$
Using claims \eqref{est.2} and \eqref{est.3} and taking $R\rightarrow\infty $ we obtain
$$\frac{\displaystyle\int_{\rr^d}\int_{\rrd} K(x,y)( u_\eps(x)- u_\eps(y))^2dxdy}
{\displaystyle\int _{\rr^d} u_\eps^2(x)dx}
 \geq \lim _{R\rightarrow \infty} \lambda_1(B_R).$$
By \eqref{eps.est}, for any $\eps>0$, we have
$\lambda_1(\rr^d)+\eps\geq \lim _{R\rightarrow \infty}
\lambda_1(B_R) $. Thus
 $$\lambda_1(\rr^d)\geq \lim _{R\rightarrow \infty}\lambda_1(B_R)  .$$
Using now \eqref{est.1}
 the proof of \eqref{lambdalimit} is finished.

It remains to prove  claims \eqref{est.2} and \eqref{est.3}.
The first claim follows from Lebesgue's dominated convergence
theorem, since $|u_{\eps,R}|\leq |u_\eps|\in L^2(\rrd)$. For the
second one we have that
$$u_{\eps,R}(x)- u_{\eps,R}(y)\rightarrow u_{\eps}(x)-u_{\eps}(y),\quad\text{as}\quad R\rightarrow \infty$$
and
\begin{equation}\label{est.5}
K(x,y)| u_{\eps,R}(x)- u_{\eps,R}(y)|^2\leq 2 K(x,y)(
u_{\eps,R}^2(x)+u_{\eps,R}^2(y))
\leq 2 K(x,y)(u^2_\eps(x)+u^2_\eps(y)).
\end{equation}
We  show that under the assumptions on $K$ the right hand side in
\eqref{est.5} belongs to $L^1(\rrd\times\rrd)$:
\begin{align*}
\int _{\rrd} \int _{\rrd} K(x,y)(u^2_\eps(x)&+u^2_\eps(y))dx dy=2\int_{\rrd} \int_{\rrd} K(x,y)u^2_\eps(x)dxdy\\
&=2\int _{\rrd}u^2_\eps(x)dx \int _{\rrd} K(x,y)dy\leq 2\sup _{x\in
\rrd} \int_{\rr^d}
K(x,y)dy \int_{\rrd} u_\eps ^2(x)dx\\
&\leq 2\int_{\R^d} \psi(x) (1+|J_{a^{-1}}(x)|)dx\int_{\rrd} u_\eps
^2(x)d \leq C \int _{\rrd} u_\eps ^2(x)dx.
\end{align*}
Applying now Lebesgue's convergence theorem we obtain
\eqref{est.3}.
\end{proof}

When we consider dilations of a domain $\Omega$ with $0
\in  \Omega$ we get the same limit. This provides a proof of Theorem~\ref{teo.conver}.

\begin{proof}[Proof of Theorem~\ref{teo.conver}] Let us
consider $B_{r_1} \subset \Omega \subset B_{r_2}$ then
$$
\lambda_1 ( R B_{r_1}) \geq \lambda_1 ( R \Omega) \geq
\lambda_1 ( R B_{r_2}),
$$
and we just observe that
$$
\lim_{R \to \infty} \lambda_1 ( R B_{r_1}) = \lim_{R \to \infty} \lambda_1 ( R B_{r_1}) =
\lambda_1 (\rr^d).
$$
This ends the proof. \end{proof}

\section{Proofs of lower and upper bounds for the first eigenvalue}
\label{main.results}
\setcounter{equation}{0}

In this section we obtain estimates on $\lambda_1(\rr^d)$ defined by
\eqref{lambda1}. First we prove Theorem~\ref{teo.cota.por.abajo}.

%\begin{theorem}\label{lem.down}
%Assume that $K(x,y)$ is of the form \eqref{forma.nucleo}
% where the
%function $a$ satisfies
% \[
% \displaystyle\sup_{x\in \rr^d}  |J_{a^{-1}}(x)|=M< 1
%\qquad \text{ or } \qquad
%  \displaystyle\inf _{x\in \rr^d}|J_{a^{-1}}(x)|=m> 1.
%\]
% Then
% $$
%\lambda_1(\rr^d)
%\geq 2(1-M^{1/2})^2 \big(\int _\rr\psi(x)dx\big).
%$$ in the first case and
%and
%$$
%\lambda_1(\rr^d)
%\geq 2(1-m^{1/2})^2 \big(\int _\rr\psi(x)dx\big).
%$$
%in the second case.
%\end{theorem}

\begin{proof}[Proof of Theorem \ref{teo.cota.por.abajo}]
First of all, let us perform the following computations: let
$\theta$ be a positive constant which will be fixed latter.
 Using the elementary inequality
 \begin{align*}
 (b-c)^2=b^2+c^2-2bc\geq b^2+c^2-\theta b^2-\frac 1\theta c^2=
 (1-\theta)(b^2-\frac {c^2}{\theta})
 \end{align*}
 we get
 \begin{align*}
 \iint_{\rr^{2d}} \psi (y-a(x))&(u(x)-u(y))^2 dxdy\geq
 (1-\theta)\iint_{\rr^{2d}} \psi(y-a(x))\Big(u^2(x)-\frac {u^2(y)}{\theta}\Big)dxdy \\
 &=(1-\theta)\Big(\int_{\rr^d}u^2(x)dx\int_{\rr^d}\psi(y)dy-\frac 1\theta
 \int_{\rr^d}u^2(y)\int _{\rr^d}\psi(y-a(x))dxdy
 \Big)\\
 &=(1-\theta)\int_{\rr^d}u^2(x)\Big(\int_{\rr^d}\psi(y)dy-\frac 1\theta\int_{\rr^d}\psi(x-a(y))dy
 \Big)dx\\
 &=(1-\theta)\int_{\rr}u^2(x)\Big(\int_{\rr^d}\psi(y)dy-\frac 1\theta\int_{\rr^d}\psi(x-y)| |J_{a^{-1}}(y)||dy
 \Big)dx\\
 &=\frac{1-\theta}\theta \int_{\rr^d}\psi(y)dy\int_{\rr^d}u^2(x)
 \Big(\theta-\frac{(\psi \ast |J_{a^{-1}}|\big)(x)}{ \int_{\rr^d}\psi(y)dy}
 \Big)dx.
   \end{align*}
Then
 \begin{align*}
\frac12 \iint_{\rr^{2d}}
 K(x,y)&(u(x)-u(y))^2 \,  dx\, dy\\
 \geq &\left\{
 \begin{array}{ll}
\displaystyle  \frac{1-\theta}\theta
\left( \int_{\rr^d}\psi(y)dy\right)\int_{\rr^d}u^2(x)\Big(\theta-\frac{\sup_{x\in \rr^d} \psi
 \ast |J_{a^{-1}}|}{ \int_{\rr^d}\psi(y)dy}
 \Big)dx,& \qquad \theta<1, \\[15pt]
\displaystyle \frac{1-\theta}\theta
\left( \int_{\rr^d}\psi(y)\, dy \right)
 \int_{\rr^d}u^2(x)\Big(\theta-\frac{\inf \psi \ast |J_{a^{-1}}|}{ \int_{\rr^d}\psi(y)\, dy}
 \Big)\, dx,& \qquad \theta>1.
 \end{array}
 \right.\\[10pt]
 \geq &
 \left\{
 \begin{array}{ll}
\displaystyle  \frac{1-\theta}\theta
\left( \int_{\rr^d}\psi(y)dy\right)\int_{\rr^d}u^2(x)\Big(\theta-M \Big)dx,& \qquad\theta<1, \\[15pt]
\displaystyle \frac{\theta-1}\theta
\left( \int_{\rr^d}\psi(y)\, dy \right)
 \int_{\rr^d}u^2(x)\Big(m-\theta
 \Big)\, dx,& \qquad \theta>1.
 \end{array}
 \right.
  \end{align*}
In the first case we choose $\theta=M^{1/2}$. In the second case
we get $\theta=m^{1/2}$.

Therefore, the statement holds from the definition of $\lambda_1
(\rr^d)$.
 \end{proof}

%\section{Upper bounds for $\lambda_1 (\rr^d)$.}
%\label{Sect.upper.bounds}
%\setcounter{equation}{0}

In the following we deal with upper bounds for the first
eigenvalue.

First, let us state a lemma with an upper bound for $\lambda_1(B_R)$
in terms of the radius of the ball, $R$, and the function $\psi$.
Note that here we are assuming that $a$ is $1-$homogeneous.

\begin{lemma}\label{lem.up}
Let $K(x,y)=\psi(y-a(x))+\psi(x-a(y))$ with an $1-$homogeneous map
$a$. For every $\delta>0$ there exists a constant $C(\delta)$ such
that the following
\begin{align*}
\lambda_1(B_R)
&\leq (2+\delta)\int _{\rr^d}\psi(z)dz\int _{\rr^d}\Big(\phi( x)-\phi( {a(x)})\Big)^2dx\\
&\quad +\frac{C(\delta)}{R^2}\int _{\rr^d}\psi(z)|z|^2dz  \int _{\rr^d} |\nabla \phi (x )|^2dx \sup _{y\in B_{1+1/R}}|J_{a^{-1}}(y) |
\end{align*}
 holds for  any function $\phi$ supported in the unit ball with
$\|\phi\|_{L^2(B_1)}=1$ and all $R>0$.
\end{lemma}

\begin{proof}
Let $\phi$ be a smooth function supported in the unit ball with $\int _{B_1}\phi^2(x)dx=1$.
Taking as a test function $\phi_R(x)=\phi(x/R)$ in the variational characterization
\eqref{lambdar}, we obtain
\begin{align*}
\lambda_1(B_R)&\leq\frac{\displaystyle\int _{\rr^d}\int_{\rrd}
K(x,y)(\phi_R(x)-\phi_R(y))^2dxdy}{\displaystyle\int _{B_R}\phi_R^2(x)dx}= \frac 1{R^d} \int_{\rr^d}\int_{\rrd} K(x,y)\Big(\phi(\frac xR)-\phi(\frac yR)\Big)^2dxdy\\
&= R^d \int_{\rr^d}\int_{\rrd} K(Rx,Ry)\Big(\phi( x)-\phi(
y)\Big)^2dxdy.
\end{align*}
Using that $K(x,y)=\psi(y-a(x))+\psi(x-a(y))$ and that the right hand side in the
last term is symmetric we get
\begin{align*}
\lambda_1(B_R)&\leq 2 R^d \int _{\rr^d}\int_{\rrd} \psi(Ry-a(Rx))\Big(\phi( x)-\phi( y)\Big)^2dxdy\\
&=2\int _{\rr^d}\int _{\rr^d}\psi(z)\Big(\phi( x)-\phi( \frac{z+a(Rx)}R)\Big)^2dxdz\\
&\leq (2+\delta)\int _{\rr^d}\int _{\rr^d}\psi(z)\Big(\phi( x)-\phi( \frac{a(Rx)}R)\Big)^2dxdz\\
&\quad+
C(\delta)\int _{\rr^d}\int _{\rr^d}\psi(z)\Big(\phi( \frac{a(Rx)}R)-\phi( \frac{z+a(Rx)}R\Big)^2dxdz\\
&\leq (2+\delta)\int _{\rr^d}\psi(z)dz\int _{\rr^d}\Big(\phi( x)-\phi( {a(x)})\Big)^2dx\\
&\quad +\frac{C(\delta)}{R^2}\int _{|z|\leq 1}\psi(z)\int
_{\rr^d}\Big(\int _0^1\nabla
\phi({a(x)}+s\frac zR)\cdot zds\Big)^2dxdz.
\end{align*}
Observe that we have
\begin{align*}
\int _{|z|\leq 1}\psi(z)\int
_{\rr^d}\Big(\int _0^1\nabla &
\phi({a(x)}+ s\frac zR)\cdot zds\Big)^2dxdz
\\ & \leq
\int _{\rr^d}\psi(z)|z|^2 \int _{\rr^d}\int _0^1|\nabla \phi \big(a(x)+\frac{sz}R\big )|^2  dsdx dz\\
&\leq \int _{\rr^d}\psi(z)|z|^2 \int _0^1 \int _{\rr^d} |\nabla \phi \big(x+\frac{sz}R\big )|^2 |J_{a^{-1}}(x)| dxds dz\\
&\leq \int _{\rr^d}\psi(z)|z|^2 \int _0^1 \int _{|x|\leq 1} |\nabla \phi \big(x\big )|^2 |J_{a^{-1}}(x-\frac{sz}R)| dxds dz\\
&\leq \int _{\rr^d}\psi(z)|z|^2dz  \int _{\rr^d} |\nabla \phi (x )|^2dx \sup _{y\in B_{1+1/R}}|J_{a^{-1}}(y) | .
\end{align*}

Hence, we have
\begin{align*}
\lambda_1(B_R)
&\leq (2+\delta)\int _{\rr^d}\psi(z)dz\int _{\rr^d}\Big(\phi( x)-\phi( {a(x)})\Big)^2dx\\
&\quad +\frac{C(\delta)}{R^2}\int _{\rr^d}\psi(z)|z|^2dz  \int _{\rr^d} |\nabla \phi (x )|^2dx \sup _{y\in B_{1+1/R}}|J_{a^{-1}}(y) |,
\end{align*}
as we wanted to show. \end{proof}

Now we are ready to prove our general upper
bound.

\begin{proof}[Proof of Theorem \ref{teo.cota.por.arriba}]
Let us fix $\delta>0$ and a function  $\phi$ supported in the unit ball with
$\|\phi\|_{L^2(\rr^d)}=1$. We apply Lemma~\ref{lem.up} and let $R\to \infty$.
Then
$$\lambda_1(\rr^d)\leq (2+\delta)\int _{\rr^d}\psi(z)dz\int _{\rr^d}\Big(\phi( x)-\phi( {a(x)})\Big)^2dx$$
Letting $\delta\to 0$ we obtain the desired result.
\end{proof}

%We have
%$$
%\begin{array}{rl}
%\displaystyle \int_{|z|\leq 1} \Big(\phi( x)-\phi(\frac{a( Rx)}{R})\Big)^2dx
%=& \displaystyle
%\int_{|z|\leq 1} \phi^2( x)dx + \int_{|z|\leq 1} \phi^2(\frac{a(
%Rx)}{R})dx\\[10pt]
%& \displaystyle  - \int_{|z|\leq 1} \phi( x)\phi(\frac{a(
%Rx)}{R})dx \\[8pt]
%\leq & \displaystyle
%\int_{\rr^d} \phi^2( x)dx + \int_{\rr^d} \phi^2(\frac{a(
%Rx)}{R})dx\\[8pt]
%\leq & \displaystyle 1 +  \int_{\rr^d} \phi^2(w)| J_{a^{-1}} (w)|
%dw\\[8pt]
%\leq & \displaystyle 1 +  M.
%\end{array}
%$$
%Therefore, we conclude that
%$$
%\lambda_1 (\rr^d) = \lim_{R\to \infty} \lambda_1 (B_R) \leq
% 2\left(\int _{\rr^d}\psi(z)dz \right) (1+M),
%$$
%as we wanted to show.

Now, we deal with the case in which $a$ is an invertible linear
map on $\rr^d$ of the form $a(x) = Ax$. To clarify the presentation  we first
treat the case of a diagonal matrix $A$. We then
 extend the result to the case of a
general matrix. The proof in the first case  is
simpler while the proof of the general case is more involved and
requires different techniques.

\begin{lemma}\label{minimizer}
Let $a(x)= Ax $  be an invertible linear map that in addition is
assumed to be diagonal, that is, $a(x) =
(\alpha_1x_1,\dots,\alpha_dx_d)^T$ with $\alpha_i \in \rr$. Then,
if we consider functions $\phi\in L^2 (\rr^d)$ supported in the
unit ball, we have
\begin{equation}
\label{liviu.ii}
 \inf _{\|\phi\|_{L^2(B_1)}=1}
\int_{\rr^d}\Big(\phi( x)- \phi( a( x))\Big)^2dx =(1-|\det
(A)|^{-1/2})^2.
\end{equation}
\end{lemma}

\begin{proof}
For any function $\phi$ as in the statement we have
\begin{align*}
\int_{\rr^d}\Big(\phi( x)-\phi( a( x))\Big)^2&=1+|\det(A)|^{-1}-2\int_{ \rr^d}\phi( x) \phi( a( x))\\
 &\geq 1+|\det(A)|^{-1}-2\left(\int_{\rr^d}\phi^2(x)dx\right)^{1/2}
 \left(\int_{\rr^d}\phi^2(a (x))dx\right)^{1/2}\\
 &=1+{|\det(A)|}^{-1}-2|\det(A)|^{-1/2}=(1-|\det(A)|^{-1/2})^2.
\end{align*}

In order to prove \eqref{liviu.ii} we need to show
the existence of a sequence of functions $\phi$ as in the statement
such that
$$
\dfrac{\displaystyle
\int_{ \rr^d}\phi( x)\phi (a( x))dx }{\displaystyle |\det(A)|^{-1/2}\int_{\rr^d}\phi^2(x)dx
}
\rightarrow 1.
$$
Choosing $\phi$ of the form (we use a standard separation of
variables here)
$$\phi(x)=\prod_{i=1}^d \phi_i(x_i){\chi}_{B_\eps}(x),\quad x=(x_1,\dots,x_d),$$
with $\eps$ small enough such that $\phi$ to be supported in the
unit ball we reduce the problem to the one dimensional case:
$a(x)=\alpha x$ and construct a sequence of functions
$\phi_\sigma$ supported in $[-\eps,\eps]$ such that
$$\frac{\displaystyle
\int_{\rr}\phi_\sigma( x)\phi_\sigma( a (x)))dx}{\displaystyle \alpha^{-1/2}\int_{\rr}\phi_\sigma^2(x)dx}
\rightarrow 1.$$

We choose
$$\phi_\sigma (x)=\frac 1{|x|^{\sigma}}{\chi}_{(0,\eps)}(x), \qquad \mbox{ with }\sigma<1/2.$$
Then
$$\int _{\rr}\phi_\sigma^2(x)dx= \int _0^\eps\frac 1{|x|^{2\sigma}}=\frac {\eps^{1-2\sigma}}{1-2\sigma}$$
and
$$
\begin{array}{l}
\displaystyle \int_{ \rr}\phi_\sigma( x)\phi_\sigma( a (x))dx  =
\int _0^{\min\{\eps,\eps/\alpha\}}\frac 1{|x|^{\sigma}}\frac 1{|\alpha
x|^{\sigma}} \\
\displaystyle \qquad \qquad = \alpha^{-\sigma}\int
_0^{\min\{\eps,\eps/\alpha\}}\frac 1{|x|^{2\sigma}}=
\frac{\alpha^{-\sigma}\min\{\eps,\eps/\alpha\}^{1-2\sigma}}{1-2\sigma}.
\end{array}
$$
Thus
$$\displaystyle\frac{\displaystyle
\int_{ \rr}\phi_\sigma( x)\phi_\sigma( a (x)))dx}{\displaystyle\alpha^{-1/2}\int_{\rr}\phi_\sigma^2(x)dx}=
\frac{\alpha^{-\sigma}\min\{\eps,\eps/\alpha\}^{1-2\sigma}}{\alpha^{-1/2}\eps^{1-2\sigma}}\rightarrow 1, \qquad
\mbox{as }\sigma\rightarrow 1/2.$$
This ends the proof.
\end{proof}

We proceed now to prove our result concerning linear functions $a$
when $A$ is diagonal.

\begin{theorem} \label{teo.lineal.ii}
Let $a(x)= Ax $  be an invertible linear map that in addition is
assumed to be diagonal, that is, $a(x) =
(\alpha_1x_1,\dots,\alpha_dx_d)^T$ with $\alpha_i \in \rr$. Then
$$\lambda_1 (\rr^d) =
\lim_{R\rightarrow \infty}\lambda_1(B_R)=2(1-|\det(A)|^{-1/2})^2\int _{\rr^d}\psi(z)dz.$$
\end{theorem}

\begin{proof}
Using the results of Theorem \ref{teo.cota.por.arriba} and Lemma \ref{minimizer}
(here we are using that $A$ is diagonal) we obtain that
 $$\lim_{R\rightarrow \infty}\lambda_1(B_R)\leq 2(1-|\det(A)|^{-1/2})^2\int _{\rr^d}\psi(z)dz.$$
On the other hand Theorem \ref{teo.cota.por.abajo} gives us that
$$\lim_{R\rightarrow \infty}\lambda_1(B_R)=\lambda_1(\rr^d)\geq 2(1-|\det(A)|^{-1/2})^2\int _{\rr^d}\psi(z)dz.$$
Thus we conclude that
$$\lim_{R\rightarrow \infty}\lambda_1(B_R)=2(1-|\det(A)|^{-1/2})^2\int _{\rr^d}\psi(z)dz.$$
and the proof is finished.
\end{proof}

Now our task is to extend the result, using different arguments to a
general lineal invertible map $a(x) =Ax$. In this case we use the
Jordan decomposition of $A$.

Recall that a linear map $a: \mathbb{R}^d \to \mathbb{R}^d$, $a(x)=Ax$ is
called expansive if the absolute value of the (complex)
eigenvalues of $A$ are bigger than one.
\begin{lemma}\label{minimizerexpansive} Let $a: \mathbb{R}^d \to \mathbb{R}^d$
be an invertible linear map.  If $a$ or $a^{-1}$ is expansive then
for functions $\phi\in L^2 (\rr^d)$ supported in the unit ball with
$ \|\phi\|_{L^2(B_1)}=1$ the following holds:
$$\sup _{\phi} \int_{\rr^d}\phi( x) \phi( a( x))dx
= |\det (A)|^{-1/2}.$$
Moreover, the supremum is
not attained.
\end{lemma}

\begin{proof}
First, given $\phi$ as in the statement, we observe that
\begin{equation}\label{kkl}
 \int_{\rr^d}\phi( x) \phi( a( x))dx \leq \left( \int_{\rr^d}\phi^2( x) dx
 \right)^{1/2} \left( \int_{\rr^d}\phi^2( a(x)) dx
 \right)^{1/2}.
\end{equation}
Hence
$$\sup_{\phi}  \int_{\rr^d}\phi( x) \phi( a( x))dx
\leq  |\det (A)|^{-1/2}.$$
Observe that in \eqref{kkl} we cannot have equality since in this case $\phi(a
(x))=\mu\phi(x)$ $a.e$ for some constant $\mu$. Since $a$ is expansive this implies that
$\phi$ should vanish identically.

Now we want to obtain the reverse inequality. Let us assume that $a$
is expansive. So, there exists $B \subset \mathbb{R}^d$ a ball with
center the origin such that $a^{-j}(B) \subset B_1$, $\forall j \in
\{ 0,1,\dots, \}$. Take the following sets
\begin{equation*}
\displaystyle F= \bigcup_{j=0}^{\infty}a^{-j}(B), \qquad E_l= a^{-l}(F) \setminus a^{-l-1}(F),
\quad \textrm{for } l \in \{0,1, \dots \}
\end{equation*}
and
\begin{equation*}
 E=  \bigcup_{j=0}E_j.
\end{equation*}
Observe that given $l \in \{0,1, \dots \}$ we have $\mid E_l \mid_d
>0$. Here and in what follows we denote by $\mid \cdot \mid_d$ the
Lebesgue measure of a set in $\rr^d$.

Since $|\det \ A| >1$, then
\begin{eqnarray*}
|a^{-l}(F)|_d =|a(a^{-l-1}(F))|_d =  |\det (a)| |a^{-l-1}(F)|_d  >
|a^{-l-1}(F)|_d.
\end{eqnarray*}
Next, let us observe that
\begin{equation} \label{Fintersection}
E_j \cap E_l= \emptyset \quad \textrm{if $j,l \in \{0,1,\dots\}$
and $j \neq  l$}.
\end{equation}
Also, since $F\supset a^{-1}(F)\supset a^{-2}(F)\supset \dots $
we have
$$|E_j|_d=|a^{-j}(F)|_d-|a^{-j-1}(F)|_d=|\det (A)|^{-j}(|F|_d-|a^{-1}(F)|_d)= |\det (A)|^{-j}|E_0|_d. $$

For any $0< \sigma<|\det (A)|^{1/2}$ we now choose
$$\phi_\sigma (x)= \sum_{j=0}^{\infty} \sigma^j \chi_{E_j}(x).$$
These functions are supported in the unit ball and belong to
$L^2(\rr^d)$, in fact,
$$
\|\phi_\sigma \|^2_{L^2(\mathbb{R}^d)}= \sum_{j=0}^{\infty} \sigma^{2j} |E_j|_d =
|E_0|_d \sum_{j=0}^{\infty} \sigma^{2j} |\det (A)|^{-j} (<
\infty).$$

On the other hand,
\begin{eqnarray*}
\int_{ \rr}\phi_\sigma( x)\phi_\sigma( a (x))dx &=&
\sum_{j=1}^{\infty} \sigma^{j-1} \sigma^j |E_j|_d = \sum_{j=1}^{\infty} \sigma^{j-1}
\sigma^j |\det (A)|^{-j} |E_0|_d \\
&=& \sigma |\det  (A)|^{-1}  |E_0|_d \sum_{j=1}^{\infty}
\sigma^{2(j-1)} |\det (A)|^{-j+1} \\ &= & \sigma |\det (A)|^{-1}
\|\phi_\sigma \|^2_{L^2(\mathbb{R}^d)}.
\end{eqnarray*}
Thus
$$\frac{\displaystyle
\int_{ \rr}\phi( x)\phi( a (x)))dx}
{\displaystyle |\det (A)|^{-1/2}\int_{\rr}\phi^2(x)dx}= \sigma
|\det (A)|^{-1/2}
\rightarrow 1, \quad \text{as}\quad  \sigma\rightarrow (|\det (A)|^{1/2})^{-},$$
which proves Lemma \ref{minimizerexpansive} in the case of an expansive function.

Assume now that $a^{-1}$ is expansive. Let  $\phi $ as in the
statement, then after the change of variable $a(x)=y$ we have
$$ \int_{\rr^d}\phi( x)\phi( a( x))dx =
|\det (A) |^{-1}\int_{\rr^d}\phi( x)\phi( a^{-1}( x))dx.
$$
 Hence, the proof finishes using the previous expansive case.
\end{proof}

\begin{lemma}\label{minimizeruni} Let $a: \mathbb{R}^d \to \mathbb{R}^d$, $a(x)=Ax$ be such that $A$ is
 diagonalizable  with  all of its complex eigenvalues having
the  absolute value equal to one. For functions $\phi\in L^2
(\rr^d)$ supported in the unit ball with $
\|\phi\|_{L^2(B_1)}=1$ the following holds
$$\sup_{\phi}  \int_{\rr^d}\phi( x)\phi( a( x))dx
= \max_{\phi}  \int_{\rr^d}\phi( x)\phi( a( x))dx = 1.$$
\end{lemma}

\begin{proof}
Take $\phi= |B_1|^{-1/2}_d\chi_{B_1}$ where $\chi_{B_1}$ is the
characteristic function of the ball with center the origin and
radius $1$. Since $\phi(a(x))= \chi_{B_1}(x)$, then the assertion
follows.
\end{proof}

\begin{lemma}\label{minimizeruni.44} Let $a: \mathbb{R}^d \to \mathbb{R}^d$, $a(x)=Ax$
be an invertible linear map such that the corresponding matrix
associated to the canonical basis is given by
\begin{equation} \label{realJordan1}
 J_k(\lambda)= \left (
\begin{array}{ccccccccc}
  \lambda & 1       &    &       \\
          &  \ddots &  & 1  \\
          &         &         & \lambda
\end{array}  \right ),
\end{equation}
or
\begin{equation} \label{complexJordan1}
 \widetilde{J}_k(\theta)= \left (
\begin{array}{ccccccccc}
  \mathbf{M} & \mathbf{I}       &    &       \\
          &  \ddots &  & \mathbf{I}  \\
          &         &         & \mathbf{M}
\end{array}  \right ),
\end{equation}
 where $\lambda\in \{\pm 1\}$, $\theta \in \mathbb{R}$, $\mathbf{M}
= \left (
\begin{array}{ccccccccc}
  \cos \theta & \sin \theta       \\
    -\sin \theta      &  \cos \theta
\end{array}  \right )
$ and $\mathbf{I} = \left (
\begin{array}{ccccccccc}
  1 & 0       \\
    0      &  1
\end{array}  \right )$.
 Then, if we consider functions $\phi\in L^2 (\rr^d)$ supported in
the unit ball with $\|\phi\|_{L^2(B_1)}=1$, we get
$$\sup_\phi  \int_{\rr^d}\phi( x)\phi( a( x))dx
= 1.$$
\end{lemma}
\begin{proof} {\bf Case I.}
Assume that the corresponding matrix of the linear map $a$
associated to the canonical basis is given by (\ref{realJordan1}).
Given $j \in
\mathbb{N} $, $a^{j}(\mathbf{p})^t= (\lambda^j + j \lambda^{j-1},
\lambda^j, 0, \dots, 0 )^t$ where $\mathbf{p}=(1, 1, 0, \dots,
0)\in \rr^d$. Observe that $a^j(\mathbf{p}) \neq a^l(\mathbf{p})$
if $l,j \in \mathbb{N}$ and $j \neq l$. Indeed $\| a^j(\mathbf{p})
- a^l(\mathbf{p}) \| \geq 1$ if $j \neq l$. Thus,
$a^j(B_{1/4}(\mathbf{p})) \cap a^l(B_{1/4}(\mathbf{p})) =
\emptyset$  if $j \neq l$, where $B_{1/4}(\mathbf{p})$ is the ball
with center the point $\mathbf{p}$ and radius $1/4$.

Given $k \in \mathbb{N}$, $k \geq 5$, set  the function
$$\phi_k(x)=
\sum_{j=0}^k \chi_{a^j(2^{-k}B_{1/4}(\mathbf{p}))}(x).
$$
Observe that the function $\phi_k$ is supported in the unit ball. If
$x$ is in the support of $\phi_k$ then there exists $j \in
\{0,1,\dots, k\}$ such that $| x - a^j(\mathbf{p}) | \leq 2^{-k-2}$
and we have
\begin{eqnarray*}
 | x  | & \leq& | x - a^j(\mathbf{p}) | + |  a^j(\mathbf{p}) | \\
 & \leq & 2^{-k-2} + 2^{-k}((1 + k )^2 + 1)^{1/2}
 \leq 2^{-k-2} + 2^{-k} 3k \leq  2^{-k+1} 3k < 1.
\end{eqnarray*}
Further,
$$
\|\phi_k\|^2_{L^2(\mathbb{R}^d)}= \sum_{j=0}^{k}  |a^j(2^{-k}B_{1/4}(\mathbf{p}))|_d
= 2^{-kd}(k+1)|B_{1/4}(\mathbf{p})|_d.
$$
On the other hand,
\begin{eqnarray*}
\int_{ \rr}\phi_k( x)\phi_k( a (x))dx &=&  \sum_{j=1}^{k}
|a^j(2^{-k}B_{1/4}(\mathbf{p}))|_d = 2^{-kd}(k)|B_{1/4}(\mathbf{p})|_d
\end{eqnarray*}
Thus
$$\frac{\displaystyle \int _{ \rr}\phi_k( x)\phi_k( a (x)))dx}{\displaystyle |\det (A)|^{-1/2}\int_{\rr}
\phi_k^2(x)dx}= \frac{k}{k+1}
\rightarrow 1, \quad k \rightarrow \infty.$$

Having in mind that $|\det (A)|=1$, that $\phi$ satisfies the
hypotheses in the statement and using H\"{o}lder's inequality we
obtain that
$$\sup_\phi   \int_{\rr^d}\phi( x)\phi( a( x))dx
\leq 1= |\det (A)|^{-1/2}.$$ Hence, the conclussion follows.

{\bf Case II.} Assume that   the corresponding matrix of $A$ in the
canonical basis is given by (\ref{complexJordan1}).
For any $j \in \mathbb{N} $ we set
\begin{equation}
 a^{j}(\mathbf{q}) = \left (
\begin{array}{ccccccccc}
 \cos (j\theta) +\sin (j\theta) + (j-1) \cos ((j-1)\theta) +(j-1)\sin ((j-1)\theta)
 \\ \cos (j\theta) -\sin (j\theta) + (j-1) \cos ((j-1)\theta) -(j-1)\sin ((j-1)\theta)      \\
         \cos (j\theta) +\sin (j\theta)  \\
         \cos (j\theta) -\sin (j\theta) \\
         0 \\
         \cdots \\
         0
\end{array}  \right ).
\end{equation}
 Observe that for $\mathbf{q} =(1,1,1,1,0\dots 0)$, we have
 $a^{j}(\mathbf{q}) \neq \mathbf{q}$ if $j \in \{ 0, \dots, k \}$,
where $k$ is a no negative integer number. So $a^j(\mathbf{q}) \neq
a^l(\mathbf{q})$ if $l,j \in  \{ 0, \dots, k \}$ and $j \neq l$.
Thus,  by continuity of the linear map $a$,  there exists $B \subset
\rr^d$ a ball with the center at  the point $\mathbf{q}$ and radius
less or equal to $1$ such that $a^j(B) \cap a^l(B) = \emptyset$  if
$j,l \in \{ 0,1, \dots, k \}$, $j \neq l$.

Given $k \in \mathbb{N}$, $k \geq 7$, we  set  the function
$$\phi_k(x)=
\sum_{j=0}^k \chi_{a^j(2^{-k}B)}(x).$$ Observe that the function $\phi_k$
is supported in the unit ball. If $x$ is in the support of
$\phi_k$ then there exists $j \in \{0,1,\dots, k\}$ such that $|
x - a^j(\mathbf{q}) | \leq 2{-k}$ and we have
\begin{eqnarray*}
 | x  | & \leq& | x - a^j(\mathbf{q}) | +
|  a^j(\mathbf{q}) | \\
 & \leq & 2^{-k} + 2^{-k}(2(2 + 2(j-1))^2 + 2^3 )^{1/2}
 \leq 2^{-k} + 2^{-k}(2(2 + 2(k-1))^2 + 2 \ 2^2 )^{1/2} \\
 &\leq&  2^{-k} + 2^{-k}(2(4(k-1))^2 + 2 (k-1)^2 )^{1/2}\\
 & \leq& 2^{-k} + 2^{-k}(2^6(k-1)^2  )^{1/2} = 2^{-k} + 2^{-k+3}(k-1)  \leq 2^{-k+4}(k-1) < 1.
 \end{eqnarray*}
Further,
$$
\|\phi_k\|^2_{L^2(\mathbb{R}^d)}= \sum_{j=0}^{k}  |a^j(2^{-k}B)|_d = 2^{-kd}(k+1)|B|_d.
$$
On the other hand,
\begin{eqnarray*}
\int_{ \rr}\phi_k( x)\phi_k( a (x))dx &=&  \sum_{j=1}^{k}
|a^j(2^{-k}B )|_d = 2^{-kd}k|B|_d
\end{eqnarray*}
Thus
$$\frac{\displaystyle \int _{ \rr}\phi_k( x)\phi_k( a (x)))dx}{\displaystyle |\det (A)|^{-1/2}
\int_{\rr}\phi^2(x)dx}= \frac{k}{k+1}
\rightarrow 1, \quad k \rightarrow \infty.$$

Now, we observe, as we did before, that
$$\sup_{\phi}  \int_{\rr^d} \phi( x)\phi( a( x))dx
\leq 1.$$
Hence, the conclusion follows.
\end{proof}

Now we are ready to proceed with the proof of our main result
concerning linear maps $a$.

\begin{proof}[Proof of Theorem~\ref{main.result}]
According to Theorem \ref{teo.cota.por.abajo},
$$\lim_{R\rightarrow \infty}\lambda_1(B_R)=\lambda_1(\rr^d)\geq 2(1-|\det (A)|^{-1/2})^2
\int _{\rr^d}\psi(z)dz.$$ So, if we prove
\begin{equation} \label{menoroigual}
\lim_{R\rightarrow \infty}\lambda_1(B_R)=\lambda_1(\rr^d)\leq 2(1-|\det (A)|^{-1/2})^2\int _{\rr^d}\psi(z)dz,
\end{equation}
the proof is finished. Let us see that (\ref{menoroigual}) holds.

 Using Jordan's decomposition, there exist $C$ and  $J$ two
$d \times d$ invertible matrices with real entries such that $A=
CJC^{-1}$. Note that $J$ is defined by  Jordan blocks, i.e,
\begin{equation} \label{Jordan}
 J= \left (
\begin{array}{ccccccccc}
  J_1 (\lambda_1)&        &    &  &        &    &           \\
          &  \ddots &  & &        &    &        \\
          &         &         & J_r (\lambda_r) &        &    &      \\
    &        &    &      &       J_{r+1} (\alpha_1, \beta_1)&        &    &       \\
    &        &    &      &       &  \ddots &  &   \\
     &        &    &      &      &         &         & J_{r+ s} (\alpha_s, \beta_s)
\end{array}  \right ),
\end{equation}
with
\begin{equation} \label{realJordan}
 J_k(\lambda)= \left (
\begin{array}{ccccccccc}
  \lambda & 1       &    &       \\
          &  \ddots &  & 1  \\
          &         &         & \lambda
\end{array}  \right ), \quad k=1,\dots, r,
\end{equation}
or
\begin{equation} \label{complexJordan}
 J_k(\alpha, \beta)= \left (
\begin{array}{ccccccccc}
  \mathbf{M} & \mathbf{I}       &    &       \\
          &  \ddots &  & \mathbf{I}  \\
          &         &         & \mathbf{M}
\end{array}  \right ), \quad k=r+1,\dots,r+s.
\end{equation}
Here $\lambda$, $\alpha$ and $\beta$ are real numbers, $\mathbf{M}
=
\left (
\begin{array}{ccccccccc}
  \alpha & \beta       \\
    -\beta      &  \alpha
\end{array}  \right )
$ and $\mathbf{I} = \left (
\begin{array}{ccccccccc}
  1 & 0       \\
    0      &  1
\end{array}  \right )$.

Given a $d_k \times d_k$ Jordan block $J_k$ as in
(\ref{realJordan}) or (\ref{complexJordan}), then either $J_k$  or
$J_k^{-1}$ is expansive, or the corresponding eigenvalue has
absolute value equal to $1$. Then by Lemma
\ref{minimizerexpansive} or Lemma \ref{minimizeruni.44} there exists
$ \{\phi_j^{(k)} \}_{j=1}^{\infty} \in L^2(\mathbb{R}^{d_k})$, $\|
\phi_j \|_{L^2(\mathbb{R}^{d_k})}=1$, a sequence of
functions supported in the unit ball of $\rr^{d_k}$ such that
\begin{equation} \label{limsup}
\lim_{j \to \infty}\int_{\mathbb{R}^{d_k} }
 \phi^{(k)}_j( x)  \phi^{(k)}_j ( J_k (x))dx=|\det \ J_k|^{-1/2}.
\end{equation}

For $j \in \mathbb{N}$, we choose
$$
\varphi_j(x_1^{(1)}, \dots, x_{d_1}^{(1)}, \dots \dots,
x_1^{(r+s)}, \dots, x_{d_1}^{(r+s)})= \prod_{k=1}^{r+s} \phi_j^{(k)}
(x_1^{(k)}, \dots, x_{d_k}^{(k)})
$$
 and
$$
\Phi_j(x)= (r+s)^{-d/4}\| C^{-1} \|^{-1/2} | \det C |^{-1/2}
\varphi_j ((r+s)^{-1/2}\| C^{-1} \|^{-1} C^{-1} x),
$$
where  $\| C^{-1} \|$ denotes the norm of $C^{-1}$ as operator on
$\rr^d$. Observe that $\Phi_j$ is supported in $B_{1}$ and $\|
\Phi_j \|_{L^2(\rr^d)}=1$ After the change of variable $\| C^{-1}
\|^{-1}C^{-1}x=y$,   we have
\begin{equation}
\label{maximizersaaaa}
\begin{array}{l}
\displaystyle \lim_{j \to \infty}\int_{\mathbb{R}^{d} }
 \Phi_j( x)  \Phi_j ( a (x))dx \\
\displaystyle  =  (r+s)^{-d/2}\| C^{-1} \|^{-1} | \det C |^{-1}
\lim_{j \to \infty}\int_{\mathbb{R}^{d} } \varphi_j((r+s)^{-1/2} \|
C^{-1}
\|^{-1} C^{-1}x)  \\
\qquad \qquad \qquad \qquad \qquad \qquad  \qquad \times
\varphi_j ((r+s)^{-1/2}\| C^{-1} \|^{-1} C^{-1}CJC^{-1} (x))dx \\
\displaystyle  =   \lim_{j \to \infty}\int_{\mathbb{R}^{d} }
 \varphi_j( y)  \varphi_j ( J y)dy  = \prod_{k=1}^{r+s}\lim_{j \to \infty}\int_{\mathbb{R}^{d_k} }
 \phi^{(k)}_j( x)  \phi^{(k)}_j ( J_k( x))dx \\
 \displaystyle = \prod_{k=1}^{r+s}|\det \ J_k(\lambda _k)|^{-1/2} = |\det (A)|^{-1/2}.
 \end{array}
 \end{equation}

 Again, using Holder's inequality, we obtain, for any function
 $\phi$ as in the statement,
\begin{equation} \label{quimicos}
\int_{\mathbb{R}^{d} }
 \phi ( x)  \phi ( a (x))dx  \leq |\det (A)|^{-1/2}.
\end{equation}
Therefore, we have
\begin{align*}
\int _{\rr^d}\Big(\phi( x)-\phi( a( x))\Big)^2&=1+|\det (A) |^{-1}-2\int_{ \rr^d}\phi( x)\phi( a( x)),
 \end{align*}
 then by (\ref{maximizersaaaa}) and \eqref{quimicos},
 $$
 \inf_{\| \phi \|_{L^2(B_1)}=1}
\int_{\rr^d}\Big(\phi( x)-\phi( a( x))\Big)^2 = (1-|\det (A)|^{-1/2})^2.
 $$
Hence, using the results contained in Lemma \ref{lem.up} the proof
is finished.
\end{proof}

 {\bf Acknowledgements.}

L. Ignat partially supported by grants PN-II-ID-PCE-2011-3-0075
and PN-II-TE 4/2010 of the Romanian National Authority for
Scientific Research, CNCS--UEFISCDI, MTM2011-29306-C02-00, MICINN,
Spain and ERC Advanced Grant FP7-246775 NUMERIWAVES.

J. D. Rossi and A. San Antolin partially supported by Supported by
DGICYT grant PB94-0153 MICINN, Spain. J. D. Rossi also
acknowledges support from UBA X066 (Argentina) and CONICET
(Argentina).

\bigskip

%%%%%%%%%%%%%%%%%%%%%%%%%%%%%%%%%%%%%%%%%%%%%%%%%%%%%%%%%%%%%%%%%%%%%
%%%                       references                              %%%
%%%%%%%%%%%%%%%%%%%%%%%%%%%%%%%%%%%%%%%%%%%%%%%%%%%%%%%%%%%%%%%%%%%%%

\end{document}